\documentclass[12pt]{amsart}
\usepackage{amsmath}
\usepackage{amssymb}
\usepackage{amsfonts}
\usepackage{mathrsfs}
\usepackage{graphicx}
\usepackage{epsfig}
\usepackage[T1]{fontenc}
\numberwithin{equation}{section}       % Number formulas within sections

\theoremstyle{plain}
\newtheorem{theorem}{Theorem}[section]

\newtheorem{prop}{Proposition}[section]

\newtheorem{coro}[prop]{Corollary}
\newtheorem{lemma}[prop]{Lemma}

\newtheorem{definition}[prop]{Definition}

\newtheorem{remark}[prop]{Remark}
\theoremstyle{remark}

\newtheoremstyle{citing}% name
  {3pt}%      Space above, empty = `usual value'
  {3pt}%      Space below
  {\itshape}% Body font
  {}%         Indent amount (empty = no indent, \parindent = para indent)
  {\bfseries}% Thm head font
  {.}%        Punctuation after thm head
  {.5em}%     Space after thm head: " " = normal interword space;
  {\thmnote{#3}}% Thm head spec

\theoremstyle{citing}
\DeclareMathAlphabet{\mathpzc}{OT1}{pzc}{m}{it} % Zapf Chancery math alphabet

\newcommand{\C}{\mathbb{C}}

\renewcommand{\H}{\mathbb{H}}
\newcommand{\N}{\mathbb{N}}

\newcommand{\R}{\mathbb{R}}

\newcommand{\teta}{\widetilde{\teta}}

\newcommand{\eps}{\varepsilon}
\begin{document}

\title[]{A simple proof of Sullivan's complex bounds}

\author{Genadi Levin}

\address{Institute of Mathematics, The Hebrew University of Jerusalem, Givat Ram,
Jerusalem, 91904, Israel}

\email{levin@math.huji.ac.il}

\date{\today}

\maketitle

\begin{abstract}

About 35 years ago Dennis Sullivan proved a precompactness property ("complex bounds") for infinitely renormalizable real quadratic polynomials with bounded combinatorics.
%which was a key step towards establishing Feigenbaum's conjectures.
We present a simple "soft" proof of this remarkable result.
% relying on a recent preprint [G. Levin, Polynomial-like dynamics of analytic maps, arXiv:2508.17308].

\end{abstract}

\section{Introduction}
A real quadratic map $f(x)=x^2+c$  of the real line $\R$ into itself is renormalizable of period $q\ge 2$ if there exists a
symmetric closed $q$-periodic interval
$I=I(q)$, that is, $f^q(I)\subset I$, $f^q: I\to I$ is unimodal, i.e., a folding map of an interval with a single turning point (which will usually assumed to be at $0$), and one of the end points $\beta$ of $I$ is a fixed point of $f^q$.
%We always assume that $0\in f^q(I)$ (high return, see Remark ??? below), which implies that $(f^q)'(\beta)>1$. Note that $\beta>0$ (i.e., $I=[-\beta, \beta]$) if $0$ is a local minimum of $f^q$ and $\beta<0$ if $0$ is local maximum of $f^q$.
%Note that the finite sequence of iterated intervals $f^i(I)$, $i=0,2,...,q-1$, are pairwise disjoint except, perhaps their end points.
To every such $f^q: I\to I$ one associates a {\it renormalization} (rescaling)
$$R^q(f)=\beta^{-1}\circ f^q\circ \beta:[-1,1]\to [-1,1].$$

The map $f:\R\to\R$ is called infinitely renormalizable with combinatorics bounded by $N$,
if $f$ is renormalizable of periods $2\le q_1<...q_n<q_{n+1}<...$ and $a_n:=q_{n}/q_{n-1}\le N$ for all $n$.
It admits an infinite decreasing sequence of symmetric periodic intervals $I(q_n)\to \{0\}$ such that $f^{q_n}: I(q_n)\to I(q_n)$ is unimodal, and the corresponding sequence $R^{q_n}f:[-1,1]\to [-1,1]$ of renormalizations.
% and $f^{q_{n+1}}=(f^{q_n})^{a_{n+1}}$ on $\R$, in particular, on $I(q_{n+1})$.
See \cite{MS}, \cite{mcm} for the renormalization theory in one-dimensional and complex dynamics.

Recall \cite{DH} that a {\it polynomial-like (PL)} map is a triple $(V', V, g)$ where $V', V$ are topological disks in the plane
%(i.e., simply connected bounded domains in the plane)
such that
$\overline{V'}\subset V$ and $g: V'\to V$ is a proper holomorphic map of some degree $d\ge 2$.

%A nice feature of a PL map $g: V'\to V$ is its stability: a small perturbation $(\tilde V, \tilde g)$ of the pair $(V, g)$ is again a PL map $\tilde g: \tilde g^{-1}(\tilde V)\to \tilde V$.

\begin{theorem}\label{t-su} (Sullivan \cite{Su1}, \cite{MS})
Given $N$ there exists $m=m(N)>0$ as follows. Let $f$ be an infinitely renormalizable real quadratic polynomial with combinatorics
bounded by $N$.
Then, for every $n\ge n(f)$,
%there are positive constants $M=M(f)$, $m=m(N)$ as follows. The
the corresponding renormalization $R^{q_n}(f):[-1,1]\to[-1,1]$ admits a polynomial-like extension $R^{q_n}(f): V_n'\to V_n$ of degree $2$ such that
$\mod(V_n\setminus\overline{V_n'})\ge m$, for the modulus of the annulus $V_n\setminus\overline{V_n'}$.
%Moreover, this bound is "beau": for all $n\ge n(f)$, $M(f)\ge m(N)>0$.
\end{theorem}
This theorem is proved in \cite{Su1} in a more generality, for maps from Epstein's class, see Definition \ref{d-eps}. Our proof also works in
this class.

Theorem \ref{t-su} was a key step in Sullivan's proof \cite{Su1} that bounded type renormalizations converge to the fixed renormalization orbit,
%a Feigenbaum's conjecture on convergence of renormalizations,
as well as in later
%proving the rest
proofs \cite{mcm1}, \cite{ly1} of other Feigenbaum's conjectures \cite{feig}.
%in Sullivan's approach to
%in proving Feigenbaum's conjectures \cite{Su1}, \cite{mcm1}, \cite{ly1}.
Since \cite{Su1}, complex bounds
%was
have been obtained, essentially, for all real analytic
infinitely renormalizable maps
%for all infinitely renormalizable real unicritical polynomials
%for real quadratic infinitely renormalizable maps of arbitrary combinatorics
%in
\cite{LS}, \cite{grsw},
%(including real analytic unimodal maps of any even criticality),
\cite{lyym},
%(following \cite{Ly}),
%\cite{grsw},
%see the Introduction of \cite{LS} for the history.
%It was generalized further in \cite{?????} to all infinitely renormalizable
%and real polynomials with simple real critical points
\cite{shen}, \cite{css}, and
%These results were
crucially used e.g. in \cite{LS}, \cite{grsw}, \cite{ly}, \cite{kss1}, \cite{kss2}.

%The proofs of the complex bounds in the works above are very technical.
%We propose a completely different approach.

%This theorem is proved in \cite{Su1} in more generality, for maps from Epstein's class, see Definition \ref{d-eps}. The proof in this note also holds in the Epstein class.

It is usually a big deal to obtain complex bounds, and proofs
%of complex bounds
in the works cited above are very technical.
We propose a new approach which
%In the present paper we give a new proof of Theorem \ref{t-m} which is comparably simple and much less technical. It
is based on the main result of \cite{L}, more precisely, on its very particular case as follows:
\begin{theorem}\label{t-bi}
Let $g:U_{-1}\to U_0$ be a proper holomorphic map of degree $d\ge 2$ where $U_{-1}, U_0$ are open simply connected subsets of $\C$ such that $U_{-1}\subset U_0$.
Suppose that $X\subset U_0$ is a compact and full (i.e. $\C\setminus X$ is connected) set which is completely invariant, i.e.,
$g^{-1}(X)=X$, and contains all critical points of $g$. Then there exists a polynomial-like restriction $g:V'\to V$ where $V', V$ are neighborhoods of $X$.
\end{theorem}
\iffalse
\begin{remark}\label{r-bi}
We don't assume a priori that $X$ is connected although this holds a posteriori as $X$ turns out to be the non-escaping set
of a PL maps with non-escaping critical points.
However,
%as it is easy to see,
in our application of Theorem \ref{t-bi} the set $X=\hat{\mathcal{J}}$ is connected (as it is easy to see beforehand from the construction of the set $\mathcal{J}$, see Sect. \ref{s-m}). In this case (of connected $X$) the proof of Theorem \ref{t-bi}
%assuming $X$ is connected
becomes significantly simpler.
%2. Methods of \cite{L} should imply that $g^{-i}(K)\to X$ as $i\to +\infty$ for any compact $K$ s.t. $X\subset K\subset U_0$.
%We don't assume that $X$ is the non-escaping set of $g:U_{-1}\to U_0$ (i.e., $X=\cap_{i\ge 0}g^{-i}(U_0)$),
%and it is not clear that this is always the case.
%2. We also don't assume a priori that $X$ is connected although this holds a posteriori as $X$ turns out to be the non-escaping set
%of a PL maps with non-escaping critical points.
%2. In \cite{L}, the conclusion of Theorem \ref{t-bi} is proved in much more generality.
%(the assumption $U_{-1}\subset U_0$ is relaxed by assuming only the existence of infinitely many pullbacks of $U_0$).
\end{remark}
\fi
The idea of using Theorem \ref{t-bi} to prove
%Theorem \ref{t-su}
complex bounds is roughly as follows. It is well-known \cite{Su1} that any limit of renormalizations $R^{q_n}f$ is a unimodal map
$g:[-1,1]\to [-1,1]$ which extends to a map $g: U_{-1}\to U_0$ from the Epstein class, and the convergence is uniform on compacts in $U_{-1}$ assuming we start with $f$ from the Epstein class.
%, see Lemma \ref{l-eps} and Proposition for a precise statement.
Here $g: U_{-1}\to U_0$ is a proper analytic map of degree $2$, where $U_{-1}, U_0$ are as in Theorem \ref{t-bi}
%$\C_J$ is a slit complex plane and $\Omega\in\C_J$ is a $U_{-1}\subset U_0$ are simply connected domain
(in fact, $U_0$ is a slit complex plane, see below). Let $T$ be a minimal $g$-completely invariant set which contains the interval $[-1,1]$.
%(in particular, $T$ contains the forward orbit of the critical point of $g$).
The main goal is to prove that the closure $\mathcal{J}$ of $T$ is a proper subset of $U_0$. Then Theorem \ref{t-bi} applies to the topological hull $\widehat{\mathcal{J}}$ of $\mathcal{J}$, the union of $\mathcal{J}$ with all connected components of its complement, and we end up with
a PL restriction $g:V'\to V$ in a neighborhood of $\widehat{\mathcal{J}}$.
%, by the main result of \cite{L},
%there exists a polynomial-like map restriction $g:V'\to V$ in a neighborhood of
%$\mathcal{J}$. $\hat{\mathcal{J}}$.
Since PL maps are stable under small perturbations, a nearby PL map persists
for any nearby (to $g$) renormalization $R^{q_n}f$, and the modulus of this PL map is close to the limit modulus
$\mod(V\setminus \overline{V'})$.
%of the PL restriction $g:V'\to V$.
That would mean complex bounds.
\begin{remark}\label{r-bi}
We don't assume a priori that the compact $X$ in Theorem \ref{t-bi} is connected although this holds a posteriori as $X$ turns out to be the non-escaping set
of a PL maps with non-escaping critical points.
However, it can be seen beforehand, from the construction, see Sect. \ref{s-m}, that
%as it is easy to see,
%in our application of Theorem \ref{t-bi}
the set $X=\widehat{\mathcal{J}}$ as above is connected.
%(as it is easy to see beforehand from the construction of the set $\mathcal{J}$, see Sect. \ref{s-m}),
In this case (of connected $X$) the proof of Theorem \ref{t-bi}
%assuming $X$ is connected
becomes significantly simpler, see \cite{L}.
\end{remark}
We realize the idea described above by giving a detailed proof
%for maps of Epstein's class which are infinitely renormalizable with bounded combinatorics.
for infinitely renormalizable maps with bounded combinatorics.
To be more precise, we prove that $\mathcal{J}$ is proper inside of $U_0$ by a contradiction. Assuming the contrary, first we show that $\mathcal{J}$ must meet slits of the slit complex plane $U_0$, i.e., it is impossible for $\mathcal{J}$ to be unbounded without touching the slits. This part of the proof holds for any combinatorics.
On the other hand, the bounded combinatorics along with real bounds imply that the function $g$ lies at the bottom of an infinite tower of maps from the Epstein class
(cf. \cite{mcm1}, see Theorem \ref{t-m} and Remark \ref{r-tower}). This allows us to extend $g:U_{-1}\to U_0$ to a holomorphic function in bigger simply connected domains, somewhat similar to \cite{EL}. That leads to a contradiction, see the proof of Theorem \ref{t-m}. Then Theorem \ref{t-su} is reduced to Theorem \ref{t-m} with help of (basically, known) Proposition \ref{p-infren-tower}.

{\it Notations}.

For $a,b\in\R$, $<a,b>$ is the closed interval with end points $a,b$ (where $a>b$ is possible),

$\R_+=\{x>0\}$, $\R_-=\{x<0\}$,

$\H^+=\{\Im(z)>0\}$, $\H^-=-\H^+$, the upper and lower half-planes,

$\Pi=\{\Im(z)>0, \Re(z)>0\}$, $\overline\Pi=\{\Im(z)\ge 0, \Re(z)\ge 0\}$,

%$Q(z)=z^2$,

for an open interval $A\subset\R$,
$\C_{A}=(\C\setminus \R)\cup A$, a slit complex plane.

{\bf Acknowledgment.} The author thanks Sebastian van Strien for helpful comments.

\section{Preliminaries: the Epstein class}\label{s-eps}
\begin{definition}\label{d-eps} (Epstein classes $E_\beta(J)$, cf. \cite{E}, \cite{Su1}, \cite{MS})
Let $g: [-1,1]\to [-1,1]$ be a real analytic map as follows: (1) $g$ is unimodal: $g'(x)=0$ iff $x=0$, (2) $g(-1)=g(1)=1$
%$g'(1)>1$,
and $0\in g([-1,1])$, i.e., $g(0)\le 0$. Given a symmetric w.r.t. $0$ open interval $J$, we say that $g\in E_1(J)$ if
there exists another open symmetric interval $J'$ such that $[-1,1]\subset J'\subset J$, $g$ extends to a real-analytic unimodal map $g:J'\to J$ and there exists a representation
$g=F\circ Q$ where $Q(z)=z^2$ and $F: [0,1]\to [g(0),1]$ is a diffeomorphism such that $F^{-1}$ extends to a univalent (i.e. holomorphic injective)
map defined on $\C_J$.
%there exist an open symmetric interval $J=(-u, u)$, $u>1$,
%and another interval $J'\subset\R$, and a diffeomorphism $F: J'\to J$ such that $F^{-1}$ extends to a univalent (i.e., holomorphic injective) map
%$F^{-1}: \C_{J}\to \C$ and the following representation holds: $g=F\circ Q$ where $Q:z\mapsto z^2$.
More generally, given $\beta\in\R\setminus \{0\}$,
% and a symmetric interval $J$, such that $[-|\beta|, |\beta|]\subset J$,
$g\in E_{\beta}(J)$ (or simply $E_\beta$) if $g=\beta g_1 \beta^{-1}$ for some $g_1\in E_1(\beta^{-1}J)$.
\end{definition}
\begin{lemma}\label{l-eps}
1. Let $g\in E_1(J)$. Then:

1a: $g:[-1,1]\to [-1,1]$ extends to a proper holomorphic map $g=F\circ Q:\Omega\to \C_J$ of degree two with the only critical point at $0$, where $\Omega=\Omega_g$ is
a simply connected domains such that $\Omega\subset \C_J$, $\Omega\cap\R=J'$, $\Omega$ is symmetric w.r.t. $\R$ and $i\R$.
Moreover, $\{g^{-i}(\C_{J})\}_{i=0}^\infty$ is a decreasing sequence of connected simply connected domains.
%The postcritical set $P_g=\overline{\cup_{i\ge 1}g^i(0)}$ of $g:\Omega\to \C_J$ is contained in the closed interval $[-1,1]$.
Furthermore, $F^{-1}$ maps the upper $\H^+$ (lower $\H^-$) half-plane into itself and $F^{-1}(J)\subset\R$,
%$$F^{-1}(D([-1,1],\theta))\subset D(F^{-1}([-1,1]),\theta),\mbox{ for every }\theta\in (0,\pi),$$
%Besides,
$F^{-1}(1)=1$ and $F^{-1}(g(0))=0$,
$g(\bar z)=\overline{g(z)}$ and $F^{-1}(\bar z)=\overline{F^{-1}(z)}$.

1b: if $g:\Omega\to\C_J$ has an attracting or parabolic cycle, then its immediate basin of attraction contains $0$.

1c: $g'(1)>1$.

2. The space $E_1(J)$ is compact: if $g_n=F_n\circ Q\in E_1(J)$, $n=1,2,...$, then there exists a subsequence $(n_j)$ such that $F_{n_j}^{-1}:\C_J\to\C$ converges uniformly on compacts to a univalent map $F^{-1}:\C_J\to \C$, moreover,
$g:=F\circ Q\in E_1(J)$ and $g_{n_j}$ converges to $g$ uniformly on compacts in $\Omega_g$.

3. More generally, if $\beta_n\to\beta\neq 0$, $J_n\to J$ where $J_n$, $J$ are symmetric open intervals such that $<-\beta,\beta>\subset J$, and if $g_n\in E_{\beta_n}(J_n)$, then
$g_{n'}\to g$ along a subsequence $(n')$ uniformly on compacts in $\Omega_g$, for some $g\in E_\beta(J)$.
\end{lemma}
\begin{proof}
1a. $\Omega_g=Q^{-1}(F^{-1}(C_J))$ is connected simply connected because the critical value $g(0)\in\C_{J}$.
As $\Omega_g\cap\R=g^{-1}(J)\cap\R=J'\subset J$, then $\Omega_g\subset\C_J$. By this and
since $g^i(0)\in [-1,1]\subset\C_J$ for all $i\ge 0$, $g^{-i}(\C_J)$ is a decreasing sequence of connected simply connected domains.

1b: as $\Omega_g\subset\C_J$, the proof is a repetition of classical arguments, see e.g. \cite{Mib}.

1c: assume $0\le g'(1)\le 1$. $g'(1)=0$ is ruled out as $g:[-1,1]\to[-1,1]$ is unimodal. If $g'(1)\le 1$ then, by 1b, $0$ must belong to a real open subinterval $A$ of $[-1,1]$ with an end point $1$ such that $g^k(x)\to 1$ as $k\to\infty$
for all $x\in A$.
Hence, $g(0)\neq 0$, that is, $g(0)<0$. But then $g$ must have a fixed point $b\in (0, 1)$. As $b\notin A$, the interval $A$ cannot contain $0$, a contradiction.

2. Define a sequence of univalent in $\C_{J-1}$ maps as follows: $B_n(w)=F_n'(1)(F_n^{-1}(w+1)-1)$.
(Here and below $A+a=\{x+a: x\in A\}$, in particular, $J-1$ is an open interval containing $[-2,0]$.) Then $B_n(0)=0, B_n'(0)=1$. This normalised  family of univalent maps
$B_n: \C_{J-1}\to\C$ is a compact family w.r.t. uniform converges on compact sets. Note also that
$B_n(w_n)=-1$ where $w_n=g_n(0)-1\in [-2,-1]$. Choose a subsequence such that $w_{n_j}\to w_0\in [-2,-1]$ and $B_{n_j}\to B$ uniformly on compacts where $B$ is univalent on $\C_{J-1}$. Since $w_0\neq 0$, then $B(w_0)\neq 0$. Therefore, $B_{n_j}(w_{n_j})=-F_{n_j}'(1)\to
B(w_0)\neq 0$, thus $F_{n_j}^{-1}(z)\to F_*^{-1}:=1-B(w-1)/B(w_0)$ uniformly on compacts in $\C_J$.
Let's show that $g_*:=F_*\circ Q\in E_1(J)$. Indeed, for each $n$, $J_n'=Q^{-1}(F_n^{-1}(J))\cap \R=Q^{-1}(F_n^{-1}(J_n\cap [g_n(0), +\infty)))\subset J_n$, hence, passing to a limit, $J':=Q^{-1}(F_*^{-1}(J))\cap \R\subset J$ and is symmetric w.r.t. $0$,
and $g_*:J'\to J$ is unimodal. Also, $0\ge g_{n_j}(0)\to g_*(0)$ and $g_*(-1)=g_*(1)=F_*(1)=1$. Thus $g_*\in E_1(J)$.
%Clearly, $g_{n_j}'(1)\to g_*'(1)\ge 1$.
%Assume that $g_*'(1)=1$. By part 1b,  $0$ must belong to a real open subinterval $A$ of $[-1,1]$ with an end point $1$ such that $g_*^k(x)\to 1$
%for all $x\in A$.
%In particular, $g_*^(0)\to 0$, hence, $g_*(0)\neq 0$, that is, $g_*(0)<0$. But then $g_*$ must have a fixed point $b\in (0, 1)$. As $b\notin A$, the interval $A$ cannot contain $0$, a contradiction.
The proof of 3. is a minor variation of the above.
\end{proof}
\section{Towers of maps from the Epstein class}\label{s-towers}
%Sullivan's Theorem \ref{t-su} will be an immediate consequence of the following Theorem \ref{t-m} along with (basically, known) Proposition \ref{p-infren-tower}, see below.
%Next theorem states roughly that the complex bounds always hold for the map at the bottom of a tower of maps from the Epstein class. Cf. \cite{mcm}
%where the notion of "tower" in the context of PL maps was introduced and developed.
\begin{theorem}\label{t-m}
Let $\mathcal{G}=\{g_n\}_{n=0}^\infty$ where $g_n\in E_{\beta_n}(J_n)$, $n=0,1,2,...$, and $\beta_n$, $J_n$ are as follows:
\begin{enumerate}\label{e-t}
%$\{\beta_n\}_{n=0}^\infty\in \R$,
\item $\beta_0=1$, $|\beta_n|<|\beta_{n+1}|$ for $n=0,1,2,,,$, and $\beta_n\to\infty$,
\item if $I_n=<-\beta_n,\beta_n>$, then $I_n\subset J_n\subset I_{n+1}$, $n=0,1,2,...$,
%$J_n=(-u_n,u_n)$, so that $|\beta_n|<u_n$ (i.e., $I_n:=[-|\beta_n|, |\beta_n|]\subset J_n$).
%Assume that:
%\begin{enumerate}\label{e-t}
%\item if $J_n=(-u_n, u_n)$ and $I_n:=[-|\beta_n|, |\beta_n|]\subset J_n$, then
%$u_n<|\beta_{n+1}|$
%(i.e.,
%$\overline{J_n}\subset I_{n+1}$, $n=0,1,...$,
\item for every $n>0$ there exists an integer $p_n\ge 2$, such that
$$g_0(x)=g_n^{p_n}(x), \mbox{ for all } x\in [-1,1].$$
\end{enumerate}
Then there exists a simply connected neighborhood $V\subset C_{J_0}$ of $[-1,1]$, such that $\overline{V'}\subset V$ where $V'=g_0^{-1}(V)$.
In other words, $g_0: V'\to V$ is a PL restriction of the map $g_0:\Omega_{g_0}\to \C_{J_0}$.
%of degree $2$.

Moreover, assuming
\begin{equation}\label{label-m}
K_1\le \frac{|J_n|}{|I_n|}\le K_2, \
%\frac{|I_{n+1}|}{|I_n|}\le \lambda,
|I_n|\le\Lambda^n, \ p_n\le N^n,
\end{equation}
for some $1<K_1<K_2<\infty$, $\Lambda>1$, $N\ge 2$ and all $n$, the following uniform bounds holds:
there exists $m=m(K_1,K_2,\Lambda,N)>0$
such that $\mod(V\setminus \overline{V'})\ge m$, for any $\mathcal{G}$ as above and for a choice of the domain $V$.
\end{theorem}
The proof will be given in Sect. \ref{s-m}.
\begin{remark}\label{r-tower}
%Although we don't use it explicitly, In applications,
In our application of Theorem \ref{t-m}, $p_{n+1}/p_n\ge 2$ and are integers (where $p_0=1)$, moreover, $g_n=g_{n+1}^{p_{n+1}/p_n}$ on $I_n$, $n=0,1,...$.
\footnote{It seems, the converse is also true: under the conditions of Theorem \ref{t-m}, $p_{n+1}/p_n\ge 2$ and are integers and $g_n=g_{n+1}^{p_{n+1}/p_n}$ on $I_n$, but we don't need this for the proof.}
In other words, $\mathcal{G}$ is a one-sided infinite tower of maps from the Epstein class. Cf. \cite{mcm1}
where the notion of tower in the context of PL maps was introduced and greatly developed. Theorem \ref{t-m} claims that
a tower of Epstein's maps is a tower of PL maps.
\end{remark}
%A real quadratic map $f(z)=z^2+c$, $-2<c<0$, is called renormalizable of period $q\ge 2$ if there exists a (symmetric closed ) $q$-periodic interval
%$I=I(q)$, that is, $f^q(I)\subset I$, $f^q: I\to I$ is unimodal, and one of the end points $\beta=\beta_q$ of $I$ is a fixed point of $f^q$.
%Our proof of Sullivan's theorem \ref{t-sul} will be a corollary of Theorem \ref{t-m} and the following, basically, known Proposition %\ref{p-infren-tower}
Now we associate to every infinitely renormalizable real quadratic polynomial with bounded combinatorics a tower as above, see Proposition \ref{p-infren-tower}. The main ingredient is the following real bounds for such maps.
%\iffalse
%see Proposition \ref{p-real}.
%Note that the real bounds are known to hold in a much more generality, see \cite{Su1}, \cite{MS}, in particular, for infinitely renormalizable maps of bounded combinatorics from the Epstein class.
First, let $I=<-\beta,\beta>$ be a periodic interval of period $q>1$ of a real quadratic map $f$ where $\beta=\beta_q$ is a fixed point of $f^q$.
%We always assume that $0\in f^q(I)$ (high return, see Remark ??? below), which implies that $(f^q)'(\beta)>1$.
Then $\beta>0$ if $0$ is a local minimum of $f^q$ and $\beta<0$ if $0$ is local maximum of $f^q$.
%Note that the finite sequence of iterated intervals $f^i(I)$, $i=0,2,...,q-1$, are pairwise disjoint except, perhaps their end points.
We always assume that $0\in f^q(I)$ (high return, see Remark \ref{r-high}), which implies that $(f^q)'(\beta)>1$.
%Let $P=\overline{\cup_{i>0}f^i(0)}$ be the postcritical set of $f$.
%Let $J=(-j,j)$ be any open symmetric interval which contains the $q$-periodic interval $I$ and is disjoint with $P\setminus I$.
%It is easy to see that $f^q\in E_\beta(J)$, moreover, if $f^q: J'\to J$ is the unimodal extension then $\overline{J'}\subset J$ (and not only $J'\subset J$).
%Indeed,
Let $W_c$ be the maximal interval that contains $f(I)\ni c$ such that $f^{q-1}:W_c\to W$ is a homeomorphism onto its image $W=W(I)$.
Then $I\subset W$,
Let $J=(-j,j)\subset W$ be a symmetric interval containing $I$.
%Then $J\supset I$.
It is easy to see that $f^q\in E_\beta(J)$, moreover, if $f^q: J'\to J$ is the unimodal extension then $\overline{J'}\subset J$ (and not only $J'\subset J$).
Indeed, let $\hat h: J\to J_c\subset W_c$ be the inverse branch of $f^{q-1}$.
%There exists a branch $\hat h:=f^{-(q-1)}$ which maps $J$ homeomorphically onto an interval $\hat J=[\hat j^c, \hat j]$, $\hat j^c<c<\hat j$.
Let $\hat J'=\hat J-c$.
Then $F_J^{-1}:=\hat h-c: J\to \hat J'$ is a diffeomorphism which extends to a univalent function on $\C_J$.
Therefore,
$f^q=(f^{q-1}+c)\circ Q=F_J\circ Q$ on the interval $J':=Q^{-1}(\hat J')\cap\R$.
To show that $\overline{J'}\subset J$, assume the contrary, i.e., $j\le j'$ where $J'=(-j', j')$, and let, for definiteness, $\beta>0$. Then $f^q([\beta,j']=[\beta, j]\subset [\beta,j']$. Since $f^q$ is monotone increasing on $[0, j']$ and $(f^q)(\beta)>1$, then there must exist another fixed point $b\in(\beta, j]$
which attracts its left semi-neighborhood $(b-\epsilon, b)$ and $(f^q)'(b)\in (0,1]$. Then $b$ is a limit point of iterates of $0$, hence,
some iterate $b_0$ of $b$ lies in $I$ and also a fixed point of $f^q$ with the multiplier in $(0,1]$, a contradiction.
%the latter contains iterates of $0$. A contradiction, since $J\cap P\subset I$.
Thus indeed $\overline{J'}\subset J$ and $f^q\in E_\beta(J)$.
%\fi

Let $f$ be infinitely renormalizable with bounded (by $N$) combinatorics and the sequence of corresponding periodic intervals $I(q_n)=<-\beta_{q_n},\beta_{q_n}>$, $n=1,2,...$.
%where $\beta_n$ is a repelling fixed point of $f^{q_n}$.
\begin{remark}\label{r-high}
As all periodic points of $f$ are repelling, the return is high: $0\in f^{q_n}(I(q_n))$, for every $n$.
\end{remark}
For each $n$, let $W_n\supset I(q_n)$ be the open interval as above. Note that $I(q_n)\subset W_n\subset I(q_{n-1})$.
%Note that $f^{q_{n+1}}=(f^{q_n})^{a_{n+1}}$.
%on $\C$, in particular, on $I(q_{n+1})$.
%\begin{remark}\label{r-high}
%Since all periodic points of $f$ are repelling, then $0\in f^{q_n}(I(q_n))$, for every $n$.
%\end{remark}
%Note that $f^{q_{n+1}}=(f^{q_n})^{a_{n+1}}$.
%Denote $P_n=P\cap I(q_n)$.

Given a real interval $A\ni 0$ and a number $t>1$, let $t.A=\{tx: x\in A\}$, the $(t-1)$-neighborhood, or $t$-enlargement of $A$.
\begin{prop}\label{p-real} (Real bounds \cite{Su1}, \cite{MS})
%Let $f$ be infinitely renormalizable with $N$-bounded combinatorics and the sequence of periodic intervals $I(q_n)$, $n=1,2,...$.
There exist an absolute constant $L>1$ and some $1<\mu<\lambda<\infty$ which depend only on $N$, such that for every $n\ge n(f)$:
(a) if $J(q_n)=L.I(q_n)$, then $J(q_n)\subset W_n$,
%$$\tilde f_n:=f^{q_n}|_{I(q_n)}\in E_{\beta_n}(J(q_n)),$$
(b) $\mu\le I(q_{n})/I(q_{n+1})\le \lambda$.
%There exist $1<\mu(f)<\lambda(f)<\infty$ and $L(f)>1$ such that, for every $n$, (a) $\mu(f)<I(q_{n})/I(q_{n+1})<\lambda(f)$, (b) the enlarged interval $L(f).I(q_n)$ is disjoint with $P\setminus I(q_n)$.
%Moreover, the previous bounds are {\it "beau"}: there exist $1<\mu<\lambda<\infty$, $L>1$
%and $K>0$
%which depend only on $N$ such that
%the above bounds hold with those constants for all $n\ge n(f)$.
%$L I(q_n)$ is disjoint with $P\setminus P_n$ and $K(q_n)\le K$,
%for all $n\ge n(f)$.
\end{prop}
Note that these real bounds are known to hold in a much more generality, see \cite{Su1}, \cite{MS}, in particular, for infinitely renormalizable maps of bounded combinatorics from the Epstein class.
%In the sequel $\mu, \lambda, L$ {\it always stand for the "beau" constants} as in Proposition \ref{p-real}. They depend solely on the bound for the combinatorics $N$.
%Let $f$ be an infinitely renormalizable quadratic polynomial with $N$-bounded combinatorics and the sequence of periodic intervals
%$I(q_n)=[-|\beta_n|, |\beta_n|]$, $n=1,2,...$.
%For each $n\ge n(f)$, let $J(q_n)=L.I(q_n)$.
% where $L=L_N$ is a "beau" constant from Proposition \ref{p-real}.

By Proposition \ref{p-real} and the preceding discussion,
%$J(q_n)\subset W(I_n)$, then
$$\tilde f_n:=f^{q_n}|_{I(q_n)}\in E_{\beta_n}(J(q_n)).$$
\begin{prop}\label{p-infren-tower}
%Let $f$ be infinitely renormalizable as above?????????
To every sequence $\mathcal{R}\subset\N$ one can associate its subsequence $S$ and a sequence of maps
$g_m\in E_{b_m}(J_m)$, where $J_m=(-L|b_m|, L|b_m|)$, $m=0,1,2,...$, as follows:
\begin{enumerate}\label{e-t-subs}
\item $b_0=1$, i.e., $g_0\in E_1(J_0)$, and the sequence of renormalizations $\{R^{q_j}f\}_{j\in S}\to g_0$ uniformly on compacts in $\Omega_{g_0}$,
%\item $\mu\le \frac{|b_{m+1}|}{|b_m|}\le \lambda$ for $m=0,1,2,,,$,
\item if $I_m=<-b_m,b_m>$, then $\mu\le \frac{|I_{m+1}|}{|I_m|}\le \lambda$ and $J_m\subset I_{m+1}$, $m=0,1,2,...$,
%$J_n=(-u_n,u_n)$, so that $|\beta_n|<u_n$ (i.e., $I_n:=[-|\beta_n|, |\beta_n|]\subset J_n$).
%Assume that:
%\begin{enumerate}\label{e-t}
%\item if $J_n=(-u_n, u_n)$ and $I_n:=[-|\beta_n|, |\beta_n|]\subset J_n$, then
%$u_n<|\beta_{n+1}|$
%(i.e.,
%$\overline{J_n}\subset I_{n+1}$, $n=0,1,...$,
\item for every $m>1$ there exists an integer $p_m\ge 2$, such that
$$g_0(x)=g_m^{p_m}(x), \mbox{ for all } x\in [-1,1].$$
\end{enumerate}
\end{prop}
\begin{proof}
For a fixed $m>n(f)$, define a finite string of rescalings
$$\mathcal{G}_m=\{g_{m,n}=\beta_m^{-1}\circ \tilde f_n\circ \beta_m: n=m,m-1,...,n(f)\}.$$
In other words, we rescale all $\tilde f_n$, $n=m,m-1,...,n(f)$, by a single rescaling $z\mapsto \beta_m z$, which turns $I(q_m)$ into $[-1,1]$
so that $g_{m,m}=R^{q_m}f$.
Thus, if $J_{m,n}=\beta_m^{-1} J(q_n)$, $\beta_{m,n}=\beta_m^{-1}\beta_n$, then
$$g_{m,m}\in E_1(J_{m,m}), \ \ g_{m,m-k}\in E_{\beta_{m,m-k}}(J_{m,m-k}), \ k=1,...,m-n(f).$$
Denote $I_{m,n}=<-\beta_{m,n},\beta_{m,n}>$ and $J_{m,n}= L.I_{m,n}$. Then
$$I_{m,m}=[-1,1], \ I_{m,n}\subset J_{m,n}\subset I_{m,n-1}, n=m,m-1,...,m-n(f).$$
Recall that $a_n=q_n/q_{n-1}$. Hence,
$$g_{m,n}(x)=g_{m,n-1}^{a_n}(x), x\in I_{m,n},\mbox{ and }
g_{m,m}(x)=g_{m,n}^{p_{m,n}}(x), x\in I_{m,m}=[-1,1],
$$
where $p_{m,n}=\frac{q_m}{q_n}=a_m a_{m-1}...a_{n+1}, \ n=m-1,...,n(f)$. Also:
\begin{equation}\label{e-p}
2^{m-n}\le p_{m,n}\le N^{m-n},
\end{equation}
\begin{equation}\label{e-I}
\mu^{m-n}<|\beta_{m,n}|<\lambda^{m-n}, \mbox{ i.e., } \mu^{m-n}<|I_{m,n}|<\lambda^{m-n}.
\end{equation}
Fix any infinite sequence of indices $\mathcal{R}\subset\N$. Apply Lemma \ref{l-eps}(2) to a sequence $\{g_{m,m}\in E_1((-L,L)): m\in\mathcal{R}\}$ and find a converging subsequence $g_{m,m}\to g_0$ along some $\mathcal{R}_1\subset\mathcal{R}$,
where $g_0\in E_1((-L,L))$.
%with extra conditions that $I_{m,m-1}\to I_{-1}$ and $p_{m,m-1}=p_1$
%(which is possible in view of (\ref{e-p})-(\ref{e-I})). Observe that $I_{-1}\supset L.[-1,1]=$ and $p_1\ge 2$.
Then apply Lemma \ref{l-eps}(3) to $\{g_{m,m-1}\in E_{\beta_{m,m-1}}(J_{m,m-1}): m\in\mathcal{R}_1\}$ and find $\mathcal{R}_2\subset\mathcal{R}_1$ such that $\beta_{m,m-1}\to b_1$ and $p_{m,m-1}=p_1$
(which is possible in view of (\ref{e-p})-(\ref{e-I})), and
$g_{m,m-1}\to g_1$. Observe that $I_1:=<-b_1,b_1>\supset L.[-1,1]=[-L,L]$, $p_1\ge 2$, $g_1\in E_{b_1}((-L|b_1|, L|b_1|))$ and
$g_0=g_1^{p_1}$ on $[-1,1]$. Then consider $\{g_{m,m-1}\}_{m\in\mathcal{R}_2}$,
pass to a subsequence $\mathcal{R}_3\subset\mathcal{R}_2$, and so on. Cantor's diagonal procedure finishes the proof.
\end{proof}
\section{Preliminaries: Poincare's neighborhoods}\label{s-po}
As in \cite{Su1}, we use Schwartz's lemma for maps of the slit complex plane and the corresponding Poincare's neighborhoods
$D(A, \theta)$.
We state here some known related to this results which are used in the proof of Theorem \ref{t-m}.

Given an interval $A\subset\R$ and an angle $\theta\in (0,\pi)$, let
$D(A, \theta)$ be the union of two Euclidean disks, symmetric to each other w.r.t. $\R$, whose boundaries intersect at the boundary of $A$ at the external angle $\theta$ with the real axes.
This is the set of points of $\C_A$ whose distance to $A$ w.r.t. the hyperbolic metric of $\C_A$ is at most $k(\theta)=\log\tan(\pi/2-\theta/2)$.
Therefore, if $\Psi: \C_{A}\to \C_{A'}$ is univalent such that $\Psi(A)=A'$, then
$\Psi(D(B,\theta))\subset D(\Psi(B), \theta)$, for every subinterval $B\subset A$ and every angle $\theta\in (0,\pi)$.

The second result is the following general property which is proved in \cite{LS}, Lemma 15.1. Recall that $Q(z)=z^2$.
\begin{prop}\label{p-ls} (see \cite{LS})
Let $K > 1$. There exists $\theta_0=\theta_0(K)>0$ such that, for all $\theta\in (0,\theta_0)$ the boundaries of $Q(D((-1,1), \theta))$
and $D((-K, 1),\theta)$ intersect each other at a point $Z(K, \theta)\in\H^+$ and at its complex conjugate. Furthermore, $Z(K,\theta)\to K^2$ as
$\theta\to 0$. Hence, the difference $D((-K, 1),\theta)\setminus Q(D((-1,1),\theta))\to [1,K^2]$ as $\theta\to 0$.
\end{prop}
\section{Proof of Theorem \ref{t-m}}\label{s-m}
%Notations:
%Denote $\Omega_n=\Omega_{g_n}$.
%, i.e., $g_n:\Omega_n\to\C_{J_n}$ a proper map of degree $2$.
%All we know is that $\Omega_n=g_n^{-1}(\C_{J_n})$ is a connected simply connected domain which is contained in $\C_{J_n}$.
%E.g., a priori, the boundaries of $\Omega_0$ and $\C_{J_0}$ could meet.
%$I_n=[-|\beta_n|, |\beta_n|]$.
Define inductively: $T_0=[-1,1]$ and $T_{n+1}=g_0^{-1}(T_n)$, $n=0,1,2,...$. As $g_0([-1,1])\subset [-1,1]$, then
$T_n=\cup_{i=0}^n g_0^{-i}([-1,1])$ so that $\{T_n\}$ is an increasing sequence of compact subsets of $\C_{J_0}$. Let
$$T=\cup_{i=0}^\infty g_0^{-i}([-1,1]),$$
which is just the smallest completely invariant w.r.t. $g_0: \Omega_0\to \C_{J_0}$ set
(i.e., $g_0^{-1}(T)=T$) which contains $[-1,1]$.
%For each $n$, $g_0(0)\in T_n$, hence, by induction, $T_{n+1}$ is connected. Thus $T$ is connected as well.

%Note that $T$ is connected because $0\in T$.
%$T\subset\cap_{i\ge 0}g_0^{-i}(\C_{J_0})$.

Let $\mathcal{J}=\overline{T}$, the closure of $T$ (in $\C$).
A priori, $\mathcal{J}$ can intersect the finite boundary $\R\setminus J_0$ of $ \C_{J_0}$ and/or be unbounded. Our goal is to prove that $\mathcal{J}$ is compactly inside of $\C_{J_0}$, i.e., $\mathcal{J}\subset\C_J$ AND $\mathcal{J}$ is bounded.
%Note that in this case $\mathcal{J}$ is $g_0$-completely invariant.
To this end, let us first show that:
%if $\mathcal{J}$ is not compactly inside of $\C_{J_0}$ then $\mathcal{J}$ must meet the real exes at some point $x>1$.
\begin{lemma}\label{l-alt}
If $\mathcal{J}$ is not compactly inside of $\C_{J_0}$ (i.e., either unbounded or intersects the boundary $\R\setminus J_0$ of $\C_{J_0}$), then $\mathcal{J}$
must meet the real axes at some point of $(1,+\infty)$, i.e.,
there exists a sequence $\{z_n\}\subset T\cap\H^+$ and a point $1<x_*<\infty$ such that $z_n\to x_*$.
\end{lemma}
\begin{proof}
%We use Poincare neighborhoods $D(A,\theta)$, $0<\theta<\pi$.
%and the following contraction property of the map $F^{-1}:\C_J\to\C_{J'}$,?????????????? see e.g. \cite{MS} (it holds because $[-1,1]\subset J$):
%$$F^{-1}(D([-1,1],\theta))\subset D(F^{-1}([-1,1]),\theta),\mbox{ for every }\theta\in (0,\pi),$$
Assume that $\mathcal{J}$ is not compactly in $\C_{J_0}$. Then, for every integer $n>0$, $T\setminus \overline{D((-1,1),1/n)}\neq\emptyset$. As $T_0\subset \overline{D((-1,1), 1/n)}$,
it follows that for some $m=m_n$,
$T_m\subset \overline{D((-1,1),1/n)}$ while $T_{m+1}=g_0^{-1}(T_m)\not\subset \overline{D((-1,1),1/n)}$. Therefore, for each $n>0$
there exist a point $z'_n\in \overline{D((-1,1), 1/n)}\cap T$ and a point $z_n\in g_0^{-1}(z'_n)$ such that $z_n\in T\setminus \overline{D((-1,1),1/n)}$. By symmetry, one can assume that
$z_n\in \H^+$.

Let $g_0=F_0\circ Q$ where $F_0^{-1}:\C_{J_0}\to\C$ is univalent. Let
$F_0^{-1}([-1,1])=[-K, 1]$.
%Then $F_0^{-1}:\C_{(-1,1)}\to \C_{(-K,1)}$ is univalent, hence,
%and the following contraction property of the map $F^{-1}:\C_J\to\C_{J'}$,?????????????? see e.g. \cite{MS} (it holds because $[-1,1]\subset J$):
Therefore,
$$F_0^{-1}(D((-1,1),\theta))\subset D((-K,1),\theta),\mbox{ for every }\theta\in (0,\pi).$$
%see \cite{Su1}, \cite{MS}.
If $K\le 1$, then $g_0^{-1}(\overline{D((-1,1), \pi/2)})\subset Q^{-1}(\overline{D((-1,1),\pi/2)})=\overline{D((-1,1),\pi/2)}$, hence,
$\mathcal{J}\in \overline{D((-1,1), \pi/2)}$, i.e., (a) holds, a contradiction. Thus, $K>1$ and
%Lemma 15.1 of \cite{LS}
Proposition \ref{p-ls} applies. We obtain that $z_n\in \{Q^{-1}(\overline{D((-K,1), 1/n)})\setminus \overline{D([-1,1],1/n)}\}\cap \H^+\to [1,K]$ as $n\to\infty$. Hence, passing to a subsequence, $z_n\to x\in [1, K]$ and the convergence is tangential, i.e., $0<\arg(z_n-1)<1/n\to 0$.
If $x\neq 1$, we let $x_*=x$. If $x=1$ and since $g_0'(1)>1$ and $T$ is forward invariant, replace each $z_n$ with $n$ big enough by some
$\tilde z_n=g_0^{k_n}(z_n)$, the first exit of the forward iterates $g_0^k(z_n)$ from a fixed small disk $B(1,\delta)$ centered at $1$. As $z_n\to 1$ tangentially, then
$\tilde z_n\to x_*\ge 1+\delta>1$.
\end{proof}
Denote $\Omega_n=\Omega_{g_n}$, $n=0,1,...$.
\begin{lemma}\label{l-ext}
For every $n>0$, $\Omega_0\subset g_n^{-q_n}(\C_{J_n})$ where the latter is a simply connected domain which contains the interval $I_n$.
Furthermore,
$g_0=g_n^{q_n}$ on $\Omega_0$ and $g_0:\Omega_0\to\C_{J_0}$ extends to a holomorphic map $g_n^{p_n}: g_n^{-q_n}(\C_{J_n})\to\C_{J_0}$.
%In particular, $g_0:[-1,1]\to [-1,1]$ extends to a locally analytic map of $\R$ into itself.
\end{lemma}
\begin{proof}
As $J_n\subset J_{n+1}$, then $\C_{J_n}\subset \C_{J_{n+1}}$. Besides, $g_0=g_n^{p_n}$ near $[-1,1]\subset\Omega_0$ and $\Omega_0$ is connected.
It follows that $g_0=g_n^{p_n}$ in $\Omega_0$ and $\Omega_0\subset g_n^{-q_n}(\C_{J_n})$. The latter domain is simply connected by
Lemma \ref{l-eps} (1a).
%Now, $g_n:\Omega_n\to \C_{J_n}$ is a proper map of degree $2$, $\C_{J_n}$ is simply connected and all iterates $g_n^i(0)$, $i\ge 0$, are well-defined as $g_0:[-1,1]\to [-1,1]$.
%Hence, for every $i>0$, $g_n^{-i}(\C_{J_n})$ is a connected simply connected domain. In particular, so is $g_n^{-q_n}(\C_{J_n})$.
%As $g_n(I_n)\subset I_n$ then $I_n\subset g_n^{-q_n}(\C_{J_n})$ and since $I_n\nearrow\R$, the locally analytic extension $g_0:\R\to \R$ is well-defined.
\end{proof}
\begin{lemma}\label{l-inf}
%Assuming $\Omega_0$ is unbounded,
Let
$$g_0^{-1}(\infty)=\{z\in\C: \mbox{ there exist } z_n\in\Omega_0, z_n\to z, \mbox{ such that } g_0(z_n)\to\infty\}.$$
Then $g_0^{-1}(\infty)\subset\partial\Omega_0$, it is a closed subset of $\C$ and
$$g_0^{-1}(\infty)\cap\R=\emptyset.$$
\end{lemma}
\begin{proof}
The first two claims are  easy to see. Now, assume, by a contradiction, that $z_n\to x\in \R$ where $z_n\in\Omega_0$ and $g_0(z_n)\to\infty$. Fix $k$ such that $x\in I_k$. As $I_k\subset g_k^{-p_k}(\C_{J_k})$, there exists two (bounded) neighborhoods $W, W'$ of $I_k$
such that $W\subset g_k^{-p_k}(\C_{J_k})$ and $g_k^{p_k}(W)\subset W'$. On the other hand, $z_n\in W\cap \Omega_0$ for all $n$ large enough,
where $\Omega_0\subset  g_k^{-p_k}(\C_{J_k})$ and
$g_0=g_k^{p_k}$ on $\Omega_0$. Thus $g_0(z_n)=g_k^{p_k}(z_n)\to\infty$ as $n\to\infty$ while $g_k^{p_k}(z_n)\in W'$ for all $n$, a contradiction since $W'$ is bounded.
\end{proof}
A simple curve (i.e., a continuous injective map) $\gamma: A\to \Pi\cup\R_+=\{\Re(z)>0, \Im(z)\ge 0\}$, where $A$ is either $[a,b]$ ($a<b<\infty$)
or $[a,+\infty)$ will be called {\it basic} if either $\gamma(t)\in \R$ for all  $t\in A$ or there exists a closed subinterval $[a,b']\subset A$
such that $\gamma(t)\in \R$ if and only if $t\in [a,b']$. The {\it base} $B_\gamma$ of $\gamma$ is then the closed interval $\gamma([a,b'])\subset\R_+$.
In particular, $\gamma(t)\in \Pi$ outside of the base (i.e., for $t>b'$).
%If $A=[a,+\infty)$, it will be called a {\it ray}.
An {\it imaginary, or i-basic} curve is a simple curve $\gamma^{im}$ with its image in $\Pi\cup i\R_+$ such that
$\gamma^{im}(t)=i\overline{\gamma(t)}$ where $\gamma$ is basic. In particular, the "base" $i B_{\gamma}$ of $\gamma^{im}$ is a closed interval
of $i\R_+$.

Since these curves are simple, we will not always distinguish between a curve and its image.
\iffalse

Introduce the following terminology for the proof of Lemma \ref{l-domega}.
For a basic curve $\Gamma$ which is defined on a {\it finite interval}, let $\Gamma^*=\Gamma\cup\overline{\Gamma}$, the union of $\Gamma$ and its complex conjugate,
and $-\Gamma^*=-\Gamma^*$, i.e., symmetric to the set $\Gamma^*$ w.r.t. $0$ (and w.r.t. $i\R$).

Then $\Gamma\cap\R=\Gamma^*\cap\R=B_\Gamma\subset\R_+$ while $-\Gamma^*\cap\R=-B_\Gamma\subset\R_-$, where $-B_\Gamma$ is called the base interval
of $-\Gamma$.

Given two basic curves $\Gamma$, $\Sigma$, a {\it basic configuration} $\mathcal{B}$ is the union $\Gamma^*\cup(-\Sigma^*)$ which is cut along their
base intervals $B_\Gamma$, $-B_\Sigma$. That means that we split $B_\Gamma$ into two identical intervals $B_\Gamma^{\pm}$ viewing from the upper
and the lower half planes, respectively, and similar for $-B_\Sigma$. The interval in between $B_\Gamma$, $-B_\Sigma$ (i.e., between the
left end of $B_\Gamma$ and the right end of $-B_\Sigma$) is called the {\it complementary interval} of the configuration.
Finally, we also consider a shifted basic configuration $\mathcal{B}+a$, for $a\in\R$, and its shifted complementary interval (which always contains the real point $a$).

\fi
\begin{lemma}\label{l-domega}
$\partial\Omega_0\cap\overline{\Pi}$ consists of the following sets:

($\infty$): $g_0^{-1}(\infty)\cap\overline{\Pi}$,

($+$): a basic curve $\gamma_+:[j,+\infty)\to \Pi\cup\R_+$ which has the base $[j', w]$, for some $w\ge j'$ such that $g_0'(w)=0$, $g_0(w)>w$ and $g_0:(1,w)\to (1,g_0(w))$ is a diffeomorphism.
%, such that
%Moreover, $g_0:\Omega\to\C_{J_0}$ extends to a homeomorphism $g_0:\gamma_+\to [j,+\infty)$.

%as follows:

%$g_0(\gamma_+(t))=t$, $t\ge j$,

%$\gamma_+(j)=j'$, $g_0(j')=j$. Here $j: J_0=(-j,j)$, $j'$: $J'_0=(-j', j')$, $1<j'\le w$, $g_0'(w)=0$, $g_0(w)>w$ and $g_0:(1,w)\to (1,g_0(w))$ is monotone,

($im$): an i-basic curve $\gamma^{im}$.
\end{lemma}
\begin{remark}\label{r-notsect}
This lemma states essentially that $g_0^{-1}(\R\setminus J_0)$
consists of two simple curves (and their symmetric ones w.r.t. $\R$, $i\R$ and $0$) that begin at $\R_+$, $i\R_+$, respectively, whose intersections with $\R$, $i\R$ are closed intervals.
This, coupled with the previous lemmas, would be enough for proving Theorem \ref{t-m}.
\end{remark}
%Introduce
In the proof of Lemma \ref{l-domega} the following terminology is used.
%which will be used for the proof of Lemma \ref{l-domega}.
For a basic curve $\Gamma$ which is defined on a {\it finite interval}, let $\Gamma^*=\Gamma\cup\overline{\Gamma}$, the union of $\Gamma$ and its complex conjugate,
and $-\Gamma^*=-\Gamma^*$, i.e., symmetric to the set $\Gamma^*$ w.r.t. $0$.
%(and w.r.t. $i\R$).
Then $\Gamma\cap\R=\Gamma^*\cap\R=B_\Gamma\subset\R_+$ while $-\Gamma^*\cap\R=-B_\Gamma\subset\R_-$, where $-B_\Gamma$ is called the base interval
of $-\Gamma$.
Given two basic curves $\Gamma$, $\Sigma$, a {\it basic configuration} $\mathcal{B}$ is the union $\Gamma^*\cup(-\Sigma^*)$ which is cut along their
base intervals $B_\Gamma$, $-B_\Sigma$. That means that we split $B_\Gamma$ into two identical intervals $B_\Gamma^{\pm}$ viewing from the upper
and the lower half planes, respectively, and similar for $-B_\Sigma$. The interval in between $B_\Gamma$, $-B_\Sigma$ (i.e., between the
left end of $B_\Gamma$ and the right end of $-B_\Sigma$) is called the {\it complementary interval} of the configuration.
Finally, we also consider a shifted basic configuration $\mathcal{B}+a$, for $a\in\R$, along with its shifted complementary interval (it contains the point $a$).
\begin{proof}[Proof of Lemma \ref{l-domega}]
For the map $g_0:\Omega_0\to\C_{J_0}$ of degree $2$, we need to prove that the limit set
$$\{Z\in\overline{\Pi}: Z=\lim_{n\to\infty}g_0^{-1}(z_n)\mbox{ for some } z_n\in\Pi, z\to z\in \R\setminus J_0\}$$
consists of (images) of two curves $\gamma_{+}$, $\gamma^{im}$ as above.
It would be enough to prove a similar claim replacing $\R\setminus J_0$ above by $I_n\setminus J_0$ (where $I_n=<-\beta_n,\beta_n>$),
for a subsequence of $n$'s tending to $\infty$.
To this end, let us fix $n$ big. Since $g_0=g_n^{p_n}$ on $\Omega_0$, we replace $g_0$ above by
$G_n:=g_n^{p_n}:\Omega_0\to\C_{J_0}$
and need to show that the set
$$\{Z\in\overline{\Pi}: Z=\lim_{n\to\infty}G_n^{-1}(z_n)\mbox{ for some } z_n\in\Pi, z\to z\in I_n\setminus J_0\}$$
consists of (images) of two curves as in ($+$), ($im$) which are defined on finite intervals.
%restrictions onto some finite intervals of curves $\gamma_{+},\gamma^{im}$.
An advantage here is that $G_n$ has an analytic continuation
%through $I_n$
into a simply connected domain
$g_n^{-p_n}(\C_{J_n})\supset \Omega_0\cup I_n$
and has on $I_n$ only simple critical point with all critical values are again in $I_n$.
Therefore, either of two branches of $G_n^{-1}$ extends continuously to either side of the cut $I_n\setminus J_0$,
with prescribed (square-root type) singularities.
Then the claim becomes rather apparent.
%and rather straightforward.
Here is a formal proof.

%$I_0$ is a $p_n$-periodic interval of a unimodal map $g_n:I_n\to I_n$. Assume for simplicity that $\beta_n>0$.
%Then $g_n(0)<0$ and $g_n=F_n\circ Q$ where $g_n(\beta_n)=F_n(\beta_n^2)=1$, $g_n(0)=F_n(0)<0$, $F_n^{-1}:\C_{J_n}\to \C_{\hat J_n}$ univalent where $[-1,1]\subset J_n$, $J_0\supset J_0'\supset I_0$

%Let $W_c:=(w_c', w_c)$ be a maximal interval containing $g_n(I_0)\ni g_n(0)$ such that $g_n^{q_n-1}|_{W_c}$ is a homeomorphism onto its image $W:=(w',w)$. Then
%$I_0\subset J_0\subset W$ because $J_0\setminus I_0$ is disjoint with the postcritical set of $g_n:I_n\to I_n$.

%For this, we use that $g_0=g_n^{p_n}$ on $g_n^{-p_n}(\C_{J_n})$ where the latter domain is simply connected and contains $\Omega_0\cup I_n$.
%The map $g_0$ has an analytic extension through $I_n$ where it has only simple critical points with all critical values also in $I_n$.
%It follows that an appropriate preimage of $g_0^{-1}(I_n\setminus J_0=g_n^{-p_n}
%To prove the existence of required basic curves $\gamma_{\pm}$ which are defined on (infinite) rays it is enough to prove similar claims for an increasing to $R_+$ closed subinterval thanks to the
%property that $g_0:\Omega_0\to\C_{J_0}$ is a restriction of the map $g_n^{p_n}:g_n^{-p_n}(\C_{J_n})\to \C_{J_n}$, for every $n=1,2,...$, where $J_n\nearrow\R$.
%Then $[-1,1]$ is a $p_n$-periodic interval for the unimodal map $g_n:I_n\to I_n$, and the following considerations become well-known and quite straightforward.
To make notations easier, rescale $g_n\in E_{\beta_n}(J_n)$ back to a map in $E_1(\beta_n^{-1}J_n)$ so that $I_n$ turns to $[-1,1]$ using the same names $g_0,g_n, I_n, I_0, J_n, J_0$ etc for corresponding rescalings.
By this convention, $g_n=F_n\circ Q$ where $g_n(1)=F_n(1)=1$, $g_n(0)=F_n(0)<0$, $F_n^{-1}:\C_{J_n}\to \C_{\hat J_n}$ univalent where $[-1,1]\subset J_n$, $J_0\supset J_0'\supset I_0$ where
$g_n^{p_n}:J_0'\to J_0$ is unimodal, $I_0$ is a $p_n$-periodic interval of the unimodal map $g_n:I_n\to I_n$, etc.

%We are interested in an appropriate preimage $g_n^{-p_n}(I_n\setminus J_0)$, namely, when we approach $I_n\setminus J_0$ from $\H^+$.

Let $W_c:=(w_c', w_c)$ be the maximal interval containing $g_n(I_0)\ni g_n(0)$ such that $g_n^{q_n-1}: W_c\to W$ is a homeomorphism
where $W$ is an interval around $0$.
By the maximality of $W_c$, $(g_n^{q_n-1})'(w_c')=(g_n^{q_n-1})'(w_c)=0$ and
$I_0\subset J_0\subset W$.
% is disjoint with the postcritical set of $g_n:I_n\to I_n$.
Let $W_{-k}=g_n^{p_n-1-k}(W_c)$, $k=0,1,...,p_n-1$, be the corresponding backward orbit of $W$, so that $W_0=W$, $W_{-(q_n-1)}=W_c$
and $g_n: W_{-k-1}\to W_{-k}$ is a homeomorphism, for $k=0,1,...,p_n-2$. In particular, $W_{-k}$, $k=0,1,...,p_n-2$,
is a subinterval of either $(0,1)$ or $(g_n(0),0)$ while $W_{-(p_n-1)}\ni g_n(0)$.
Because of that, one can decompose the two-valued inverse map $G_n^{-1}:\C_{J_0}\to\Omega_0$ as follows:
$$G_n^{-1}=g_n^{-1}\circ \Phi$$
where $\Phi:\C_{J_0}\to\C$ is a univalent map such that $\Psi(J_0)\subset \Psi(W)=W_c$. Furthermore,
$\Psi=(g_n)_{\eps_{p_n-2}}^{-1}\circ (g_n)_{\eps_{p_n-3}}^{-1}\circ...\circ (g_n)_{\eps_0}^{-1}$
where $(g_n)_{\eps_k}^{-1}$, $\eps_k\in\{+,-\}$, is one of the two branches of $g_n^{-1}$ such that $(g_n)_{\eps_k}^{-1}(W_{-k})=W_{-k-1}$.
Similarly to $W_{-k}$, let $J_{-k}=g_n^{p_n-1-k}(J_c)$ where $J_c\ni g_n(0)$ is an interval such that $g_n^{p_n-1}:J_c\to J_0$ is a homeomorphism.
We have: $J_{-k}\subset W_{-k}$.

As $g_n^{-1}=\pm \sqrt{F_n^{-1}}$ where $F_n^{-1}$ is univalent in a neighborhood of $[-1,1]$ and $F_n^{-1}([g_n(0),1])=[0,1]$,
it is easy to see that $(g_n)_{\eps_k}^{-1}$ maps a shifted basic configuration $\mathcal{B}$ with the complementary interval $J_{-k}$ onto
another shifted basic configuration $\mathcal{B}'$ which is either in $\{\Re(z)\ge 0\}$ or in $\{\Re(z)\le 0\}$ and with the complementary interval $J_{-k-1}$. Moreover, it maps a base $B$ of $\mathcal{B}$ onto a base of $\mathcal{B}'$ if $g_n(0)\notin B$
and otherwise onto a figure of the form $\pm [0,x]\cup \pm [0,iy]$.
We begin with the basic configuration $\mathcal{B}_0$
which is $I_n\setminus J_0$. Therefore, $\mathcal{B}_1:=\Psi(\mathcal{B}_0)$ is a shifted basic configuration whose complementary interval is $J_c\ni g_n(0)$
and two basic intervals are two components of $[w_c',w_c]\setminus J_c$. Applying $g_n^{-1}$ to $\mathcal{B}_1$, we obtain the required curves,
with $w=g_n^{-1}(w_c)\cap\{\Re(z)>0\}>0$.
\end{proof}
\begin{coro}\label{c-omega-r}
$\overline{\Omega_0}\cap\R=[-w,w]$.
\end{coro}
\begin{proof}
Let $x\in\partial\Omega_0\cap\R_+$. By Lemma \ref{l-inf} and Lemma \ref{l-domega}, $x\in (\gamma_+\cup\gamma^{im})\cap\R_+$, therefore,
$x$ is in the base $B_{\gamma_+}$, i.e., $x\le w$.
%$\gamma_{\pm}$ are boundary curves of $\Omega_0$ and $g_0$ extends continuously to $\gamma_{\pm}$. Hence, if $\gamma_+(t_n)\to z\in\C$ for some $t_n\to\infty$, then $z\in g_0^{-1}(\infty)$. And similarly for $\gamma_-$.
%Therefore, the claim follows from Lemmas \ref{l-inf}-\ref{l-domega} along with
By symmetry of $\Omega_0$, we are done.
\end{proof}
\begin{lemma}\label{l-T}
$\mathcal{J}$ is disjoint with $\partial\C_{J_0}$ and is bounded.
\end{lemma}
\begin{proof}
Assuming the contrary, by Lemma \ref{l-alt},
there exists a sequence $\{z_n\}\subset T\cap\H^+$ and a point $1<x_*<\infty$ such that $z_n\to x_*$.
As $T\subset\Omega_0$, $x_*\in\overline{\Omega_0}$. Then the case $x_*>w$ is ruled out by Corollary \ref{c-omega-r}.
Thus $1<x_*\le w$. The map $g_0:[1,w)\to [1,g(w))$ is a diffeomorphism. As $1$ is a repelling fixed point of $g_0$, $g_0(w)>w$
and, moreover, $g_0\in E_1$, the homeomorphism $g_0:[1,w]\to [1,g_0(w)]$ cannot have fixed points other than $1$. Therefore,
iterates of any point of $(1,w]$ must leave $(1, w]$. Hence, $g_0^k(x_*)>w$ for some $k>0$.
As $z_n\to x_*$ where $z_n\in T\subset\cup_{i\ge 0}g_0^{-i}(\Omega_0)$, all iterates $g_0^i(z_n)\in\Omega_0$, $i\ge 0$.
Thus $g_0^k(z_n)\in\Omega_0$. On the other hand, for $n\to\infty$, $g_0^k(z_n)\to g_0^k(x_*)>w$, and this case was ruled out at the beginning of the proof. A contradiction.
\end{proof}
Thus $\mathcal{J}$ is a compact subset of $U_0=\C_{J_0}$.
%Moreover, $0\in [-1,1]\subset \mathcal{J}$ and $\mathcal{J}$ is connected as the closure of the connected set $T$.
Let $\widehat{\mathcal{J}}$ be the topological hull of the compact $\mathcal{J}\subset\C$, i.e., the union of $\mathcal{J}$ with all
bounded components of $\C\setminus\mathcal{J}$. Since $U_0$ is simply connected and $\mathcal{J}\subset U_0$, then $\widehat{\mathcal{J}}\subset U_0$.
\begin{lemma}\label{l-compl-inv}
$\widehat{\mathcal{J}}$ is completely invariant:
$$g_0^{-1}(\widehat{\mathcal{J}})=\widehat{\mathcal{J}}.$$
\end{lemma}
\begin{proof}
By Lemma \ref{l-T},
$\mathcal{J}$ is compactly contained in the domain $U_0$. Therefore, as $\mathcal{J}=\overline{T}$ and $g_0^{-1}(T)=T$, by continuity,
$g_0^{-1}(\mathcal{J})=\mathcal{J}$ as well.
%, then $\mathcal{J}\subset\cap_{i\ge 0}U_{-i}$ and
%$\mathcal{J}$ is completely invariant as well.
%Since $U_{-i}$ is simply connected, $\hat{\mathcal{J}}\subset U_{-i}$ for every $i$.
Let's prove that $X:=\widehat{\mathcal{J}}$ is completely invariant, i.e., $g_0^{-1}(X)=X$. If $z\in \partial X$ then $g_0^{\pm}(z)\in X$ as
$\partial X\subset \mathcal{J}$.
Let $z\in W$, a component of $Int(X)$. As $\partial W\subset\partial X$ then $g_0^{-1}(\partial W)\subset X$, which implies that $g_0^{-1}(z)\subset X$.
As for $z_1=g_0(z)$, if we assume by a contradiction that $z_1\notin X$, then, since $U_0\setminus X$ is connected, one can join $z_1$ and $\partial U_0$ by a simple arc
$\sigma_1\subset U_0\setminus X$. Then a component $\sigma$ of $g_0^{-1}(\sigma_1)$ joins $z\in Int(X)$ with $\partial U_{-1}$. Hence,
$\sigma$ must cross $\partial X$ at some $w$. But then $g_0(w)\in \sigma_1\cap X$, a contradiction. Thus $g_0^{\pm}(W)\subset X$ as well.
\end{proof}
{\it End of the proof of Theorem \ref{t-m}.} $g_0: U_{-1}\to U_0$ is a proper analytic map of degree $2$ with a single critical point $0$, where $U_{-1}=\Omega_0\subset \C_{J_0}=U_0$.
%and all $U_{-i}=g_0^{-i}(U_0)$ are well-defined and simply connected.
As we have proved,
%$\overline{\cup_{i\ge 0}g_0^i(0)}\subset [-1,1]\subset \hat{\mathcal{J}}\subset \cap_{i\ge 0} U_{-i}$ where
$\widehat{\mathcal{J}}\subset U_0$ is a completely invariant full compact containing $[-1,1]\ni 0$.
Therefore, Theorem \ref{t-bi} applies which guarantees the required PL restriction.
%$g_0: U_{-1}\to U_0$ is a BI map and by Theorem 1.1 of \cite{L}, there exist neighborhoods $V', V$ of $\hat{\mathcal{J}}$ such that $g_0: V'\to V$ is a PL map.

The last claim about the uniform bound follows from the fact that a collection of sequences $\mathcal{G}$ as in Theorem \ref{t-m} which satisfy the inequalities
(\ref{label-m}) is compact (the proof repeats the one of Proposition \ref{p-infren-tower}).
\section{Proof of Theorem \ref{t-su}}
It follows from the uniform bound of Theorem \ref{t-m} (with $K_1=K_2=L$ and $\Lambda=\lambda$, where, in turn, $\lambda, L$ depend only on $N$) along with Proposition \ref{p-infren-tower}.
\iffalse

Details: first, there is $n_0(f)$ s.t. for each $n\ge n_0(f)$, $R^{q_n}f:[-1,1]\to [-1,1]$ extends to SOME PL map of degree $2$.

\footnote{This is well-known to be held for any $n$ if $f$ is quadratic, by means of Yoccoz's puzzles, although can be shown easily modifying
an argument??? as above}

Indeed, otherwise there would be an infinite sequens $\mathcal{R}_0$ along with there is no such an extension. The, by Proposition \ref{p-infren-tower} it contains a subsequence  $S_0$ s.t. there is a limit to the tower as in Theorem \ref{t-m} which does have such an extension, hence, nearby renormalizations also do, a contradiction with the def. of $\mathcal{R}_0$. Now, for each
$n\ge n_0(f)$, let $m_n$ be the supremum of moduli of PL extensions of $R^{q_n}f$. If there is an infinite sequence $\mathcal{R}_1$ along which
$m_n< m(L,L,\lambda)$ from Theorem \ref{t-m}, then it contains a subsequence>>.....

\fi
\begin{remark}\label{r-alldeg}
All results hold if one replaces the degree $2$ of the unimodal map $g:[-1,1]\to [-1,1]$ of the Epstein class
by any even degree $d\ge 2$. The only noticeable change is in the proof of
Lemma \ref{l-alt} where Proposition \ref{p-ls}, which is proved in \cite{LS} for the map $z\mapsto z^2$, should be coupled with Lemma 15.2 of \cite{LS} about comparison of images under maps $z\mapsto z^d$ for different even $d$.
\end{remark}

\end{document}